\documentclass[12pt]{amsart}

\usepackage{amsmath}
\usepackage{amssymb} 
\usepackage{array}
\usepackage{enumitem}
\usepackage{hyperref}
\usepackage{multicol} 
\usepackage[all]{xy}
\usepackage{verbatim} 
\usepackage{color}
\definecolor{battleshipgrey}{rgb}{0.52, 0.52, 0.51} 
\hypersetup{
       colorlinks=true,       
    linkcolor=battleshipgrey,     
    citecolor=battleshipgrey,        
    filecolor=magenta,      
    urlcolor=battleshipgrey          
}
\usepackage{pstricks-add}

\theoremstyle{plain}
\newtheorem{theorem}{Theorem}[section]
\newtheorem{lemma}[theorem]{Lemma}
\newtheorem{proposition}[theorem]{Proposition}
\newtheorem{corollary}[theorem]{Corollary}

\newtheorem{definition}[theorem]{Definition}
\theoremstyle{remark}
\newtheorem{remark}{Remark}[section]
\newtheorem{example}{Example}[section]

\newtheorem*{acknowledgment}{Acknowledgment}

\numberwithin{equation}{section}

\newcommand{\K}{\mathbb{K}}

\newcommand{\bP}{\mathbb{P}}

\newcommand{\R}{\mathbb{R}}
\newcommand{\Z}{\mathbb{Z}}
\newcommand{\N}{\mathbb{N}}
\newcommand{\PP}{\mathbb{P}}

\newcommand{\cX}{\mathcal{X}}

\newcommand{\cP}{{\mathcal P}}

\newcommand{\fS}{\mathfrak{S}}

\newcommand{\GL}{\mathrm{GL}}

\newcommand{\lcm}{\mathrm{lcm}} 

\newcommand{\Gras}{\mathrm{Gras}}

\newcommand{\op}{\mathrm{op}}

\newcommand{\true}{{\mathrm{true}}}
\newcommand{\false}{{\mathrm{false}}}
\newcommand{\rechts}{{\mathrm{right}}}
\newcommand{\links}{{\mathrm{left}}}
\newcommand{\oder}{{\mathrm{or}}}
\newcommand{\und}{{\mathrm{and}}}

\newcommand{\msk}{\medskip}
\newcommand{\ssk}{\smallskip}
\newcommand{\nin}{\noindent}

\begin{document}

\title[Distributive lattices, associative geometries]{Distributive lattices, associative geometries:\\ the arithmetic case}


\author{Wolfgang Bertram}

\address{Institut \'{E}lie Cartan de Lorraine \\
Universit\'{e} de Lorraine at Nancy, CNRS, INRIA \\
Boulevard des Aiguillettes, B.P. 239 \\
F-54506 Vand\oe{}uvre-l\`{e}s-Nancy, France\\
url: \url{http://iecl.univ-lorraine.fr/~Wolfgang.Bertram/}}

\email{\url{wolfgang.bertram@univ-lorraine.fr}}



\begin{abstract} 
 We prove an identity for five arguments, valid
  in the lattice of natural numbers with gcd and lcm as lattice operations.
 More generally, this identity characterizes 
 arbitrary distributive lattices. 
 Fixing three of the five arguments,
 we always get {\em associative} products, and thus 
every distributive lattice carries many semigroup structures.  
 In the arithmetic case, we explicitly compute multiplication tables of
 such semigroups and describe some of their properties. 
 Many of them are periodic, and can be seen as 
  ``non-commutative analogs'' of the rings $\Z/ n\Z$.
\end{abstract}

\maketitle

\section{Introduction}

For a quintuplet 
$(x,a,y,b,z)$ of elements of  a lattice $\cX$ with operations
$\land$ (meet) and $\lor$ (join), we define two other elements by
\begin{equation}\label{eqn:1}
\begin{matrix} 
L = L(x,a,y,b,z) & := & \bigl(b \land (z \lor (a \land y))\bigr) & \lor &
\bigl(z \land (b \lor (x \land y))\bigr) & \lor 
\\
&  &
 \bigl(x \land (a \lor (z \land y))\bigr) & \lor & \bigl(a \land (x \lor (b \land y))\bigr), &
\end{matrix}
\end{equation}
\begin{equation}\label{eqn:2}
\begin{matrix} 
U = U(x,a,y,b,z) & := & \bigl( a \lor (z \land (b \lor y)) \bigl)  & \land & 
  \bigl( x \lor (b \land (z \lor y)) \bigr) 
& \land 
\\
& & 
\bigl( z \lor (a \land (x \lor y)) \bigl) & \land & 
\bigl( b \lor (x \land (a \lor y)) \bigl) .
\end{matrix}
\end{equation}
The terms defining $L$ and $U$ will also be denoted  by
$L = L_1 \lor L_2 \lor L_3 \lor L_4$ and
$U = U_1 \land U_2 \land U_3 \land U_4$
(related to two other terms $L_5,U_5$, cf.\ Eqn.\ (\ref{eqn:L5U5})).
We study the maps $L, U : \cX^5 \to \cX$ thus defined for the following kinds
of lattices:
\begin{enumerate}
\item
the {\em arithmetic case}:
$\cX = \N_0$ is the lattice of natural numbers (with $0$), with
$\land = \lcm$ ({\em least common multiple}) and
$\lor = \gcd$ ({\em greatest common divisor}),
\item
the {\em totally orderd}, or {\em chain case}: 
here $M$ is a totally ordered set, with $\lor = \max$ and $\land = \min$,
\item
the {\em power set case}: $\cX = \cP(M)$ is the power set of a set $M$,
with $\land = \cap$ being intersection and $\lor = \cup$ union,
\item
the {\em Grassmannian case}: $W$ is a (right) module over a unital ring
$\K$, and $\cX$ the space of all submodules of $W$, with meet
$\land = \cap$ and join $\lor = +$.
\end{enumerate}

\begin{theorem}\label{th:Main}
In cases {\rm (1) -- (3)}, we have $U = L$, i.e., 
\begin{equation}
\label{eqn:main}
\forall x,a,y,b,z \in \cX : \qquad 
L(x,a,y,b,z) = U(x,a,y,b,z).
\end{equation}
\end{theorem}

\nin
In the Grassmannian case {\rm (4)}, we have $L \leq U$, i.e.,
\begin{equation}
\label{eqn:main'}
\forall x,a,y,b,z \in \cX : \qquad 
L(x,a,y,b,z) \subset  U(x,a,y,b,z).
\end{equation}
The starting point of the present work was the discovery, triggered by
computer checks (cf.\ Remark \ref{rk:Sage}), that (to our big surprise), in 
the arithmetic case Inequality (\ref{eqn:main'}) becomes
an equality. 
From the point of view of abstract lattice theory, this fact is explained as follows
(Theorem \ref{th:distrib}):

\begin{theorem}\label{th:distrib0}
Let $\cX$ be a  lattice. Then the following are equivalent:
\begin{enumerate}
\item[\rm (i)]
The identity $U = L$ holds in $\cX$.
\item[\rm (ii)]
The lattice $\cX$ is \href{https://en.wikipedia.org/wiki/Distributive_lattice}{\em distributive} : $\forall x,y,z \in \cX$,
$x \land (y \lor z) = (x \land y) \lor (x \land z)$.
\end{enumerate}
\end{theorem}

Moreover, we see that the lattice is {\em modular} when  Inequality (\ref{eqn:main'}) holds, and 
one may conjecture that the converse also holds   (see Remark \ref{rk:conjecture}) -- we shall come back to this problem in subequent work.

\ssk 
The most important issue about the quintary map $L=U$ thus defined is that {\em fixing three of the five arguments, it defines associative products
on $\cX$}. More precisely,  the ``central'' variable
$y$ shall be among the three fixed argments.  For instance, fixing $(a,y,b)$, we define the {\em principal product} $\bullet:\cX^2 \to \cX$ by
\begin{equation}
x \bullet z := x \bullet_{a,y,b} z := L(x,a,y,b,z) = U(x,a,y,b,z).
\end{equation}
Because of the obvious invariance of $L$ under the Klein four-group $V$
acting on the variables $(x,a,b,z) = (v_1,v_2,v_3,v_4)$ (Lemma \ref{la:Klein}), we get
$\frac{24}{4}=6$ different kinds of ``products'' on $\cX$. 
Following
a terminology used by Conway and Smith (\cite{CS}), we present these
 six products as a ``hexad'' of products, labelled by the action of the symmetric group $\fS_3 = \fS_4 / V$
 (e.g., $((23)\bullet )(x,z) = z \bullet x$ is the opposite product of $\bullet$; in general, opposite vertices correspond to
 opposite products), 
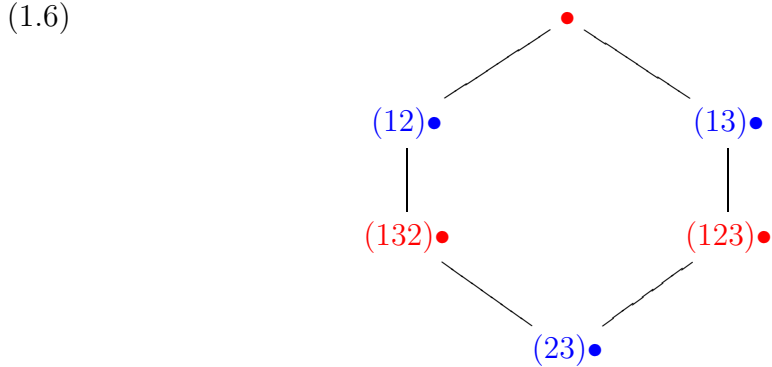
\begin{equation}\label{eqn:hexad2}
\xymatrix{ &  {\red \bullet }  \ar@{-}[rd] \ar@{-}[ld] & \\
{\blue (12) \bullet }  \ar@{-}[d] & &  {\blue (13) \bullet}  \ar@{-}[d] \\
{\red  (132)\bullet }  \ar@{-}[rd] & &  {\red (123) \bullet} \ar@{-}[ld] \\
& \blue{ (23) \bullet } & 
}
\end{equation}

\begin{theorem}
Let $\cX$ be a distributive lattice. Then the six products given by the above ``hexad'' are all  
{\em associative}. In other terms, they define semigroup structures on $\cX$.
These semigroups are {\em weak bands}, in the sense that they
satisfy the identity
$$
\forall v,w \in \cX : \qquad
v^2 w = v w = v w^2 .
$$
\end{theorem}

Instead of checking associativity by direct (and necessarily long) computation, we proceed 
by using \href{https://en.wikipedia.org/wiki/Distributive_lattice#Representation_theory}{\em representation theory for distributive lattices}: 
they can be imbedded into power set lattices (case (3) mentioned above);  and in the power set case,  we can  decompose the product
into ``atoms'', where the atoms are six elementary products called ``true, false, left, right, and, or'', and which obviously are associative and weak bands.



In the arithmetic case, we study these products further: many of them
are {\em periodic}, and then essentially reduce to {\em finite semigroups}
(Theorem \ref{th:arithmetic_case}).
In Section \ref{sec:arithmetic}, we give  several examples
of ``multiplication tables'' of such finite semigroups.
To give an idea,  the product $x\bullet_{3,2,4} z$
has ``column period'' $3$, and ``line period'' $4$:
\begin{center}
\begin{tabular}{l | l      l   l      l        l    l     l   l    l     l }  
$x\bullet_{3,2,4} z$ & 0 & 1 & 2 & 3 & 4 & 5 & 6  \\
\hline
0 & 12 & 4 & 4& 12 & 4 & 4 & 12 \\
1 & 3 & 1 & 1 & 3 & 1 & 1  & 3 \\
2 & 6 & 2 & 2 & 6 & 2&   2 &6 \\
3 & 3 & 1 & 1 & 3 & 1&  1 & 3 \\
4 &  12 & 4 & 4 & 12 & 4  & 4 & 12
\end{tabular}
\end{center}
To illustrate associativity:
$(5 \bullet 6) \bullet 6 = 3 \bullet 6 = 3 = 5 \bullet 6 = 
5 \bullet (6 \bullet 6)$.
Note that the set
$\{1,\ldots, 12\}$, as well as the set of divisors of $12$, form semigroups
for $\bullet$.
Principal products are periodic, and thus
can be seen
as analogs of the usual quotient rings $\Z / n \Z$. 
The parastrophes are not always periodic. 
One may wonder if   these semigroups have non-trivial applications
in number theory. 

\ssk
Motivation for the present work comes from joint work with Michael
Kinyon on  {\em associative geometries}, \cite{BeKi1, BeKi2, BeKi12}.
An associative geometry has as underlying space  a {\em Grassmannian} (case
(4) in the above list), 
whence an underlying lattice structure playing an important role in the theory.
The algebraic structure of an associative geometry is encoded by a quintary {\em structure map}
$\Gamma : \cX^5 \to \cX$, $(x,a,y,b,z) \mapsto \Gamma(x,a,y,b,z)$.
This structure map has  interesting algebraic properties showing a 
``geometric flavor''.  Thus it is a natural question to ask 
how it looks like in the 
{\em arithmetic case} $W = \K=\Z$, and to find a good 
algorithm for computing it. The answer is that, in this case,
 simply $L = \Gamma = U$, 
which furnishes an excellent algorithm. 
In the case of general Grassmannians, the situation is much more
complicated -- 
 we intend to study the relation between
$L,U$ and $\Gamma$
in the case of
Grassmannians, and of modular lattices, in subsequent work.


\begin{acknowledgment}
I thank Michael Kinyon for helpful comments, and for having first checked
the equivalence (i) $\Leftrightarrow$ (ii) from Theorem \ref{th:distrib0}
by using the automated theorem prover
\href{https://www.cs.unm.edu/~mccune/mace4/}{Prover9}.
\end{acknowledgment}

\section{Prelimary remarks on the general case} 

In every lattice $\cX$,
the expressions $L,U,L_1,\ldots,U_4$ defined by (\ref{eqn:1}), (\ref{eqn:2}), are
closely related to the following expressions (see the following proof for 
explanations concerning the labelling)
\begin{align}\label{eqn:L5U5}
L_5:= L_5^{(3)} := (b \land z) \lor (a \land x) , & \qquad
U_5 := U_5^{(2)} := (a \lor z) \land (b\lor x) , 
\\
L_5^{(1)} :=  (b \land a) \lor (z \land x) ,&
\qquad
U_5^{(1)}   := (a \lor b) \land (z\lor x) ,
\label{eqn:L5'}
\\
L_5^{(2)} := (b \land x) \lor (a \land z) , &
\qquad
U_5^{(3)}  := (a \lor x) \land (b\lor z) .
\label{eqn:L5''} 
\end{align}

\begin{lemma}\label{la:Klein}
In any lattice,  the
expressions
$L(x,a,y,b,z)$ and $U(x,a,y,b,z)$ defined by
 (\ref{eqn:1}) and (\ref{eqn:2})
are invariant under the action of the Klein
$4$-group (double transpositions), acting on the variables $(x,a,b,z)$, and so are the terms
$L_5,\dots,U_5^{(3)}$ defined above.
\end{lemma}

\begin{proof}
This follows immediately from the definitions. For instance, exchanging simultaneously
$(a,b)$ and $(x,z)$ exchanges $(L_2,L_3)$ and $(L_1,L_4)$, and so on; 
due to commutativity
of $\land$  and  $\lor$, the order is irrelevant.
More formally, to fix the action of the symmetric group $\fS_5$ on the five variables, we fix the
correspondence
$$
{\bf x}:= (x_1,x_2,x_5,x_3,x_4) := 
(x,a,y,b,z) .
$$
We let act
 $\fS_4$ on the variables $(x,a,b,z)=(x_1,x_2,x_3,x_4)$, and 
$\fS_3$ on $(x,a,b)=(x_1,x_2,x_3)$, in the usual way, and hence these
groups also act  on functions of these variables, like
$L_1,\ldots,U_5$.
Thus, for instance, 
$L_5^{(3)}$ is obtained by applying the transposition $(13)$ to $L_5^{(1)}$, and so on.
Note that 
the upper index $3$ in $L_5^{(3)}$ 
 indicates the variable $x_3 = b$ with which $x_4 =z$ is paired via
 $b \land z$ in Formula (\ref{eqn:L5U5}), etc.
\end{proof}

As we will see, in general lattices, $L(x,a,y,b,z)$ is in general different
from $U(x,a,y,b,z)$. However, some ``diagonal values'' of $L$ and $U$ always
agree:

\begin{theorem}\label{th:diagonals}
Let $(\cX,\land,\lor)$ be a lattice and $(x,a,y,b,z) \in \cX^5$. Then:
\begin{enumerate}
\item
for $a=z$ and $b=x$, we get
$L(x,z,y,x,z)= z \land x = U(x,z,y,x,z)$,
\item
for $b=z$ and $x=a$, we get
$L(x,x,y,z,z)= x \lor z = U(x,x,y,z,z)$,
\item
for $a=y=b$, we get
$L(x,y,y,y,z)=y=U(x,y,y,y,z)$,
\item
if $\cX$ is bounded, with maximal element $1$ and minimal element $0$, then

$L(x,1,0,1,z)= x \lor z = U(x,1,0,1,z)$,

$L(x,0,1,0,z)=x \land z = U(x,0,1,0,z)$,

$L(x,1,1,0,z)=  x = U(x,1,1,0,z)$,

$L(x,0,0,1,z) = z = U(x,0,0,1,z)$,

$L(x,a,0,b,z) = L_5(x,a,y,b,z) = (b \land z) \lor (a \land x) $,

$U(x,a,1,b,z) = U_5 (x,a,y,b,z) = (a \lor z) \land (b\lor x)$,

$L(x,0,y,b,z) = L_2(x,a,y,b,z) = z \land (b \lor (x \land y))$,

$U(x,1,y,b,z) = U_2(x,a,y,b,z) = x \lor (b \land (z \land y))$.
\end{enumerate}
\end{theorem}

\begin{proof}
All claims follow by direct computation using the defining identities
of a lattice, in particular,
$a\lor (a \land x) = a = a \land (a \lor x)$.
\end{proof}

\begin{lemma}
The function  $L$ is {\em monotonic}: if ${\bf x}\leq {\bf x}'$
(meaning that $x_i \leq x_i'$ for $i=1,\ldots,5$), then
$L({\bf x}) \leq L({\bf x}')$.
The same holds for $U,L_i,U_i$, $i=1,\ldots,5$.
Moreover, for $i=1,\ldots,4$,
$$
L_i \leq L, \qquad U \leq U_i,
$$
and when the lattice $\cX$ is bounded, this also holds for $i=5$.
\end{lemma}

\begin{proof}
Since both $\land$ and $\lor$ are monotonic operations, 
the same holds for $L$, etc.
The inequalities for $i=1,\ldots,4$ follow directly from the definition of $L$, resp.\ $U$, as join, resp.\ 
meet, of these expressions. For $i=5$, in the bounded case, this follows
by monotony from Item (4) of the preceding
theorem, since $0 \leq y$ and $y \leq 1$.
\end{proof}

Clearly, monotonic lattice morphisms induce morphisms of $L$ and of $U$. 
The following is obvious from Formulae (\ref{eqn:1}), (\ref{eqn:2}):

\begin{lemma}\label{la:anti}
If $\phi : \cX \to \cX$ is an {\em antitone lattice morphism}, i.e.,
$\phi(a \lor b) = \phi(a) \land \phi(b)$ and
$\phi(a \land b)=\phi(a) \lor \phi(b)$, then
\begin{align*}
\phi L (x,a,y,b,z) & = U(\phi x,\phi b,\phi y,\phi a,\phi z), \\
\phi U (x,a,y,b,z) & = L(\phi x,\phi b,\phi y,\phi a,\phi z).
\end{align*}
\end{lemma}

\section{Characterization of distributive lattices} 

Recall that a lattice $\cX$  is called  \href{https://en.wikipedia.org/wiki/Distributive_lattice}{\em distributive} if,  for all $x,y,z \in \cX$,
$$
x \land (y \lor z) = (x \land y) \lor (x \land z).
$$
This is equivalent to the dual identity 
$$
x \lor (y \land z) = (x\lor y) \land (x \lor z).
$$

\begin{theorem}\label{th:distrib}
For every lattice $\cX$,  the following  properties are equivalent:
\begin{enumerate}
\item[\rm (i)]
The identity $U = L$ holds in $\cX$.
\item[\rm (ii)]
The lattice $\cX$ is distributive.
\end{enumerate}
\end{theorem}

\begin{proof}
(ii) $\Rightarrow$ (i): 
If $\cX$ is distributive, 
the expressions $L_i,U_i$ can be transformed
\begin{align*}
L_1 & = (b\land (z \lor (a \land y)))= (b\land z) \lor (b \land a \land y), 
\\
L_2 & = (z \land b) \lor (z \land x \land y),\\
L_3 & = (x\land a) \lor (x \land z \land y),\\
L_4 & = (a \land x) \lor (a \land b \land y), \\
U_1 & = (a \lor (z \land (b \lor y))) = (a\lor z) \land (a \lor b \lor y) \\
U_2 & = (x\lor b) \land (x\lor z \lor y)  \\
U_3 & = (z\lor a) \land (z \lor x \lor y) \\
U_4 & = (b\lor x) \land (b \lor a \lor y) .
\end{align*}
Using this, we get, using the terms defined by  (\ref{eqn:L5U5}),
(\ref{eqn:L5'}), (\ref{eqn:L5''}),
\begin{align}
L & = (a \land x) \lor (b \land z) \lor (a\land b \land y) \lor (x \land z \land y)
= L_5^{(3)} \lor (L_5^{(1)}  \land y),
\label{eqn:A}
\\
U & = (a \lor z) \land (x \lor b) \land (a \lor b \lor y) \land (x \lor z \lor y) =  U_5^{(2)} \land (U_5^{(1)}
 \lor y) , \label{eqn:B}
\end{align}
where we have abbreviated
$L_5^{(i)} = L_5^{(i)}(x,a,y,b,z)$, and
$U_5^{(i)} = U_5^{(i)}(x,a,y,b,z)$.
These elements generate a lattice having a remarkably simple
structure:

\begin{proposition}\label{prop:4-lattice}
Fix elements $x,a,b,z$ in a distributive lattice $ \cX$. 
Then 
for $\{ i,j , k\} = \{ 1,2, 3 \}$,
writing $L_5^{(i)} = L_5^{(i)}(x,a,y,b,z)$, and
$U_5^{(i)} = U_5^{(i)}(x,a,y,b,z)$,
\begin{align*}
U_5^{(i)} \land U_5^{(j)}   = L_5^{(k)} , & \qquad
L_5^{(i)} \lor L_5^{(j)}   = U_5^{(k)} .
\end{align*}
The $6$ elements $L_5^{(i)}, U_5^{(i)}$
($i=1,2,3$)
generate a lattice of $8$ elements, which is a 
homomorphic image of  the lattice of subsets of $\{ 1,2,3\}$,
as indicated by the diagram:
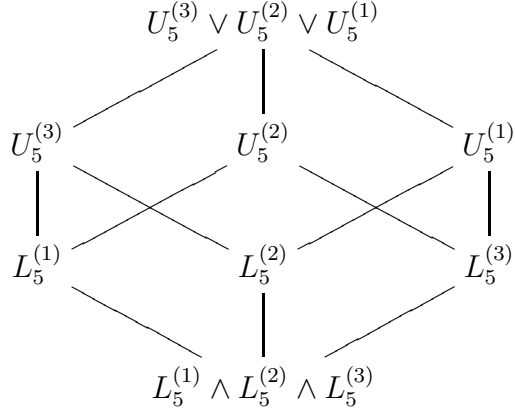
\begin{equation*}
\xymatrix{ &  U_5^{(3)} \lor U_5^{(2)}  \lor U_5^{(1)}    \ar@{-}[rd] \ar@{-}[ld]
 \ar@{-}[d]  & \\
 U_5^{(3)}     \ar@{-}[d] \ar@{-}[rd]& U_5^{(2)}  \ar@{-}[rd] \ar@{-}[ld] &  U_5^{(1)}  \ar@{-}[d] \ar@{-}[ld] \\
L_5^{(1)}   \ar@{-}[rd] & L_5^{(2)}  \ar@{-}[d]  &  L_5^{(3)}  \ar@{-}[ld] \\
&  L_5^{(1)} \land  L_5^{(2)}  \land L_5^{(3)}   & 
}
\end{equation*}
\end{proposition}

\begin{proof}
In the following, we drop the  lower index $5$. 
For instance, if $(i,j,k)=(1,2,3)$, then
by distributivity, 
$$
L^{(3)} =(b \land z) \lor (a \land x)=
(b\lor a)\land (z\lor x) \land (b\lor x) \land (z \lor a)=U^{(1)} \land U^{(2)} ,
$$
and similarly for the other relations.
It follows,
 for $\{ 1,2,3 \} = \{ i, j ,k \}$, that
\begin{align*}
L^{(i)}\land L^{(j)} & = ( U^{(j)} \land U^{(k)} ) \land ( U^{(i)} \land U^{(k)}) 
\\
&= U^{(1)} \land U^{(2)} \land U^{(3)} 
\\
& = L^{(i)} \land L^{(k)},
\end{align*}
whence
$$
L^{(i)}\land L^{(j)} = L^{(1)} \land L^{(2)} \land L^{(3)} .
$$
In the same way,
$$
U^{(i)} \lor U^{(j)} = L^{(1)} \lor L^{(2)} \lor L^{(3)} =
U^{(1)} \lor U^{(2)} \lor U^{(3)} 
$$
which implies also $U^{(i)} \land L^{(i)} = L^{(1)} \lor L^{(2)} \lor L^{(3)}$, etc. 
\end{proof}


As a particular case of the proposition, we have
 $$
 U_5^{(2)} \land L_5^{(3)}  = (L_5^{(1)}  \lor L_5^{(3)}) \land L_5^{(3)}  = L_5^{(3)} = U_5^{(1)} \land U_5^{(2)}.
$$
This is used in the fifth equality of the followig computation, along
with  distributivity and the relation $L_5^{(3)} \lor L_5^{(1)} = U_5^{(2)}$:
\begin{align*}
L & =  L_5^{(3)} \lor (L_5^{(1)}  \land y)  \\
& = (L_5^{(3)}  \lor L_5^{(1)}) \land (L_5^{(3)} \lor y) \\
& =  U_5^{(2)} \land (L_5^{(3)}  \lor  y) \\
& = (U_5^{(2)}  \land L_5^{(3)})  \lor (U_5^{(2)} \land y) \\
&  
 = (U_5^{(1)} \land U_5^{(2)}) \lor (U_5^{(2)} \land y) \\
 &
= U_5^{(2)} \land (U_5^{(1)} \lor y) = U ,
\end{align*}

\ssk
(i) $\Rightarrow$ (ii): 
Since every non-distributive lattice 
\href{https://en.wikipedia.org/wiki/Distributive_lattice#Characteristic_properties}{contains  a sublattice
isomorphic to the diamond lattice $M_3$ or to the pentagon lattice $N_5$}
(see \cite{BS}), it is enough to show that the identity $L = U$ is not satisfied
in $M_3$ and in $N_5$. 
First assume that $\cX$ is the {\em diamond lattice} 
$$
M_3 = \{ 0, u,v,w,1 \}, \quad
 u \land v = 0 = u \land w = v \land w, \quad
u \lor v = 1 = v \lor w = w \lor u,
$$
and choose
$x=z=u$, $y=v$, $a=b=w$, as indicated by the Hasse diagram:
$$
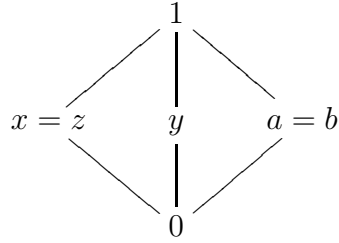
\xymatrix{ &  1  \ar@{-}[rd] \ar@{-}[d]   \ar@{-}[ld]   & \\
x=z   \ar@{-}[rd] & y \ar@{-}[d] &  a=b \ar@{-}[ld] \\
& 0 & 
}
$$
We have $a \land y =0= x \land y = z \land y = 
b \land y$, and  $a \lor y =1= x \lor y = z \lor y = 
b \lor y$,
and  (with the four terms ordered as in Equations
(\ref{eqn:1})  and
(\ref{eqn:2})
\begin{align*}
L(x,a,y,a,x) & =
(b \land z) \lor (z \land b)  \lor (x\land a)\lor (b\land x)
\\
&=0 \lor 0 \lor 0 \lor 0  = 0 , \\
U(x,a,y,a,x) & = (a \lor z)\land (x \lor b)  \land
(x\lor a) \land (b\lor x) 
\\
&=  1 \land 1 \land 1 \land 1  =1 .
\end{align*}
Thus $L=U$ does not hold in $M_3$.
Next, assume that $\cX=N_5$ is the {\em pentagon lattice},
and choose $a=y=u$, $b=w$, $z=v$, and $x$ any of the
elements of $N_5$: 
$$
N_5 = \{ 0, u,v,w, 1 \}, \qquad 0 < u < w < 1 , \quad 0 < v < 1 ,
$$
$$
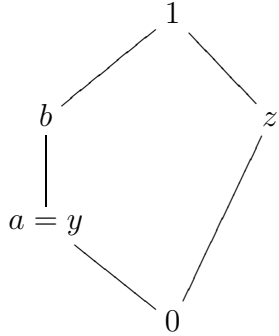
\xymatrix{ &  1  \ar@{-}[rd]  \ar@{-}[ld]   & \\
b  \ar@{-}[d]  & &  z  \ar@{-}[ddl] \\
 a=y \ar@{-}[rd] & & 
 \\
& 0 & 
}
$$
Since $L =L_1 \lor L_2 \lor L_3 \lor L_4$,  we have
$L_1 \leq L$, and likewise $U \leq U_1$. But
\begin{align*}
L_1 = b \land (z \lor (a \land y)) & = b \land (z \lor a) = b \land 1 = b ,\\
U_1 = a \lor (z \land (b \lor y)) & = a \lor (z \land b) = a \lor 0 = a,
\end{align*}
whence $U \leq U_1 = a < b = L_1 \leq L$, and hence $L$ is not equal to $U$ in $N_5$. 
\end{proof}

\begin{remark}
The characterization of  distributive lattices by the identity $L=U$ can be seen as a
``higher'' analog of the known fact (cf.\ \cite{Bi}, II., Theorem 8) that {\em $\cX$ is distributive iff it satisfies the
median law}
$$
(x\lor y) \land (y \lor z) \land (z \land x) =
(x \land y) \lor (y \land z) \lor (y \land x).
$$
\end{remark}

\begin{remark}
The cubic lattice from Prop.\ \ref{prop:4-lattice} can be seen as part of the
free distributive lattice on $4$ generators.
Its top element is the join of the 6 possible meets of the generators,
and its bottom element the meet of the 6 possible joins.
When two of the generators coincide (say, $x=a$), then we get another part,
which is just a square (say, $L^{(1)} = U^{(2)} = L^{(3)}$ and
$U^{(1)} = L^{(2)} = U^{(3)}$). This square is the product of the trivial
lattice given by the median element (preceding remark) and a lattice
with $2$ generators.
\end{remark}

\begin{corollary}\label{cor:cases}
The following identities hold:
\begin{enumerate}
\item
In the {\em arithmetic case} (lattice $\N_0$  with $\lor = \gcd$,
$\land = \lcm$)
\begin{align*}
\gcd \bigl(\lcm (a,x), \lcm(b,z),\lcm(a,b,y),\lcm(x,y,z)  \bigr)   & = 
\\
\lcm(\gcd (a,z),\gcd(x,b),\gcd(a,b,y),\gcd(x,y,z)).
\end{align*}
\item
In the {\em totally ordered case} (chain with $\lor = \max$, $\land = \min$)
\begin{align*}
\max (\min (a,x), \min(b,z),\min(a,b,y),\min(x,y,z)) & = 
\\
\min(\max (a,z),\max(x,b),\max(a,b,y),\max(x,y,z)).
\end{align*}
\item
In the {\em Boolean case} (power set $\cX=\cP(M)$
  with $\lor = \cup$, $\land = \cap $)
\begin{align*}
(a \cap x) \cup (b \cap z) \cup (a \cap b\cap y) \cup (x \cap y \cap z) &=
\\
(a \cup z) \cap (x \cup b) \cap (a \cup b \cup y) \cap (x \cup y \cup z).
\end{align*}
\end{enumerate}
\end{corollary}

\begin{proof}
All three lattices are well-known to be distributive, hence satisfy
$L = U$ in the form given by (\ref{eqn:A}) and
(\ref{eqn:B}).
\end{proof}

\begin{theorem}
Assume $\cX$ is a lattice satisfying the identity $L \leq U$. Then this lattice is
modular.
\end{theorem}

\begin{proof}
If $\cX$ is not modular, then 
\href{https://en.wikipedia.org/wiki/Modular_lattice#Examples}{it contains the pentagon lattice $N_5$ 
as sublattice} (see \cite{Bi}, I., Theorem 12), and as we have seen in
the preceding proof, $L \leq U$ does not hold in $N_5$.
\end{proof}

\begin{remark}\label{rk:conjecture}
What about  the converse: does modularity imply $L \leq U$? 
As mentioned in the introduction,
the identity $L \leq U$
does hold in Grassmannians, which form
an important class of modular lattices. 
We shall come back to this question in subsequent work.
\end{remark}

\begin{theorem}
Let $\cX$ be a distributive lattice, and
$\phi:\cX \to \cX$ antitone.
Then 
$$
\phi U(x,a,y,b,z) = U (\phi x,\phi b ,\phi y, \phi a, \phi z).
$$
In particular, this holds 
when $\cX$ is a
 \href{https://en.wikipedia.org/wiki/Boolean_algebra_(structure)}{\em Boolean algebra} (complemented distributive lattice), with complement map
 $\phi = \lnot:\cX \to \cX$.
\end{theorem}

\begin{proof}
This
follows from $U=L$, together with Lemma \ref{la:anti}.
\end{proof}

\begin{example}\label{ex:anti-auto}
The lattice $\N_0$ does not carry any anti-automorphism. But
let $N \in \N$, $N > 1$, and $K$ a divisor of $N$.
Then the distributive
lattice (interval) $\cX = [K,N] = \{ d \in \N \mid \, K\vert d, \, \, d \vert N \}$ carries
an anti-automorphism $\phi (d) = \frac{KN}{d}$.
If $\frac{N}{K}$ contains only simple prime-powers as factors, then
$\phi$ is a  complement map, but otherwise it's not. 
\end{example}

In the distributive case, the six versions 
$\sigma.L = \sigma.U$ ($\sigma \in \fS_3$) of $L=U$, organized as a ``hexad''
according to  the scheme from
(\ref{eqn:hexad2}), are explicitly given as follows, in terms of
 $L_5^{(i)}$ and $U_5^{(i)}$. 
We use that $\sigma. L_5^{(i)} = L_5^{(\sigma(i))}$, resp.\
$\sigma. U_5^{(i)} = U_5^{(\sigma(i))}$.
Note that the invariance group in $\fS_4$ of each term
$L_5^{(i)}$, resp., $U_5^{(i)}$, is the subgroup of order
$8$ generated by the Klein $4$-group together with the
transposition $(i 4)$. 
There are $3$ such subgroups, all
isomorphic to a dihedral group $D_4$.
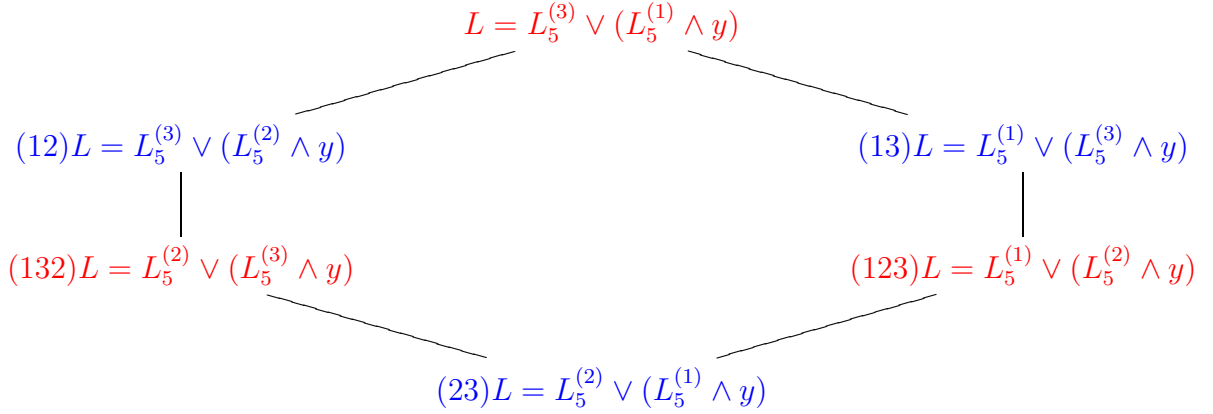
\begin{equation*}\label{eqn:hexad2bis}
\xymatrix{ &  {\red L =   L_5^{(3)} \lor (L_5^{(1)}  \land y) }  \ar@{-}[rd] \ar@{-}[ld] & \\
{\blue (12) L =  L_5^{(3)} \lor (L_5^{(2)}  \land y)  }  \ar@{-}[d] & &  {\blue (13) L =  L_5^{(1)} \lor (L_5^{(3)}  \land y)}  \ar@{-}[d] \\
{\red  (132)L =  L_5^{(2)} \lor (L_5^{(3)}  \land y)  }  \ar@{-}[rd] & &  {\red (123)L =  L_5^{(1)} \lor (L_5^{(2)}  \land y)} \ar@{-}[ld] \\
& \blue{ (23)L =  L_5^{(2)} \lor (L_5^{(1)}  \land y) } & 
}
\end{equation*}
When $L=U$,  this hexad coincides with
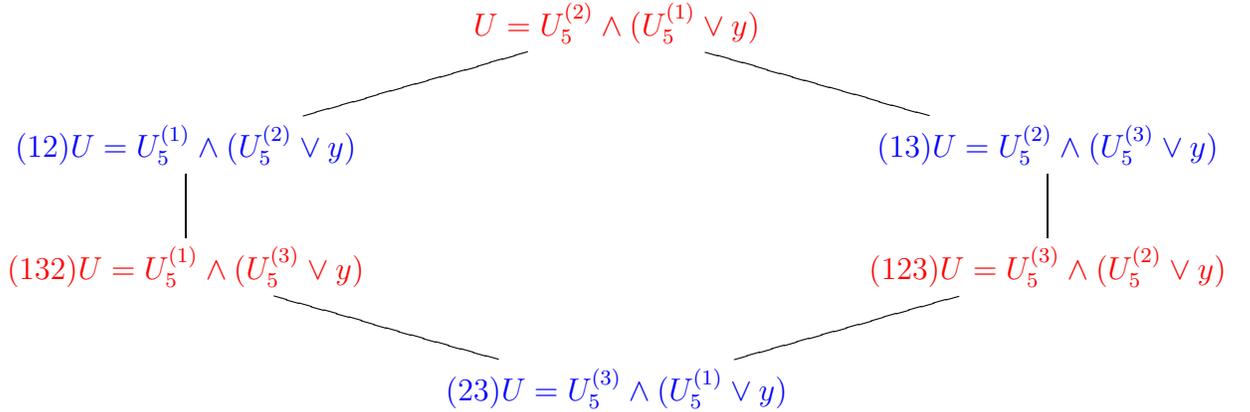
\begin{equation*}\label{eqn:hexad2ter}
\xymatrix{ &  {\red U  =   U_5^{(2)} \land (U_5^{(1)}  \lor y) }  \ar@{-}[rd] \ar@{-}[ld] & \\
{\blue (12) U =  U_5^{(1)} \land (U_5^{(2)}  \lor y)  }  \ar@{-}[d] & &  {\blue (13) U =  U_5^{(2)} \land (U_5^{(3)}  \lor y)}  \ar@{-}[d] \\
{\red  (132)U =  U_5^{(1)} \land (U_5^{(3)}  \lor y)  }  \ar@{-}[rd] & &  {\red (123)U =  U_5^{(3)} \land (U_5^{(2)}  \lor y)} \ar@{-}[ld] \\
& \blue{ (23)U =  U_5^{(3)} \land (U_5^{(1)}  \lor y) } & 
}
\end{equation*}

\begin{remark}
Under {\em duality of lattices}, i.e., exchange of $\land$ and $\lor$,
that is, exchange of $\leq$ and $\geq$,
$U_5^{(i)}$ and $L_5^{(i)}$ exchange, for $i=1,2,3$, and hence
the first hexad exchanges with the ``dual'' of the second (where
by ``dual'' hexad we mean the one obtained after applying a
central symmetry). 
When the lattice is distributive, duality thus correponds to central
symmetry.
\end{remark}

\section{A glimpse on the chain case}

The statement of Item (2) of Corollary \ref{cor:cases} can be refined:

\begin{theorem}\label{th:order}
Let $M$ be a totally ordered set,  
with $\land = \min$ and $\lor = \max$.
Then
\begin{enumerate}
\item
if $a \leq  y \leq b$, then 
$U = L = \max(a,\min(z,b)) = \min(\max(a,z),b)$,  
\item
if $b \leq y \leq a$, then
$U=L= \min(a,\max(x,b)) = \max(\min(a,x),b)$,
\item
if $x \leq  y \leq z$, then 
$U = L = \max(x,\min(b,z)) = \min(\max(x,b),z)$, 
\item
if $z \leq  y \leq x$, then 
$U = L = \min(x,\max(a,z))=\max(\min(x,a),z)$, 
\item 
if $a,b \leq y \leq x,z$, then
$U = L = y$,
\item
if $x,z \leq y \leq a,b$, then
$U=L = y$,
\item
if $a,x,b,z \leq y$, then 
$U = L = L_5 = \min(\max(b,z),\max(a,x))$,
\item
if $y \leq a,x,b,z$, then
$U = L = U_5 = \max(\min(a,z),\min(b,x))$. 
\end{enumerate}
\end{theorem}

\begin{proof}
This follows in each case by inspection of the expressions
appearing in Item (2) of Corollary \ref{cor:cases}.
(In principle, there are $5 ! = 120$ cases to be checked; due to
invariance under double transpositions, this reduces to 
$30$ cases, which can be checked by hand or by machine.) 
\end{proof}

\begin{remark}
What is the order-theoretic interpretation of these relations?
Just as the total order $\leq$ is encoded by $\min$ and
$\max$, the expression $L =U$ should encode some order-theoretic
concept. One may think of the 
\href{https://en.wikipedia.org/wiki/Cyclic_order}{\em cyclic order} 
corresponding to the total order, defined on the one-point completion
of $M$. When $M = \R$, then
this one-point completion is the projective line $\R \PP^1 = S^1$,
and $L$ is in this case related to the
{\em cross-ratio}. Thus $L$ could play the role of a
``cross-ratio type invariant for chains''.
\end{remark}

\section{Associativity}\label{sec:associativity}

In a distributive lattice $\cX$, 
 fixing a triple $(a,y,b)$, 
we define a binary ``product'' 
\begin{equation} \label{eqn:principal}
x \bullet  z := x \cdot_{a,y,b} z :=L(x,a,y,b,z)=U(x,a,y,b,z) .
\end{equation} 
We are going to show that this product is always associative, hence
defines a semigroup struture on $\cX$.
More generally, let us define:

\begin{definition}\label{def:product}
Let $\cX$ be an arbitrary lattice.
Fixing three of the five components of $\cX^5$, we define ``product maps''
$\cX^2 \to \cX$ by considering the remaining 
two components as ``variables''.
More specifically, the central component $y$ shall always belong to the ``fixed'' variables.
Explicitly,
for $(x,a,y,b,z) \in \cX^5$, we let
\begin{align*}
L^{(x,z)}_{(a,y,b)} & := L^{(z,x)}_{(b,y,a)} := x \bullet_{a,y,b}^L  z := L(x,a,y,b,z)  = L(z,a,y,b,x)
\\ 
L^{(a,z)}_{(x,y,b)}  &:=  L^{(z,a)}_{(b,y,x)} := L(x,a,y,b,z) = L(a,z,y,x,b),
\\
L^{(b,z)}_{(x,a,y)} & := L^{(z,b)}_{(a,x,y)} :=  L(x,a,y,b,z) = L(a,x,b,z,b) .
\end{align*} 
Retaining only the subscripts, this is represented by the  ``hexad'' of binary products
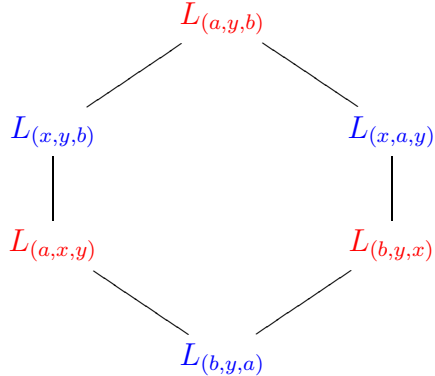
\begin{equation*}\label{eqn:hexad}
\xymatrix{ &  {\red L_{(a,y,b)}}  \ar@{-}[rd] \ar@{-}[ld] & \\
{\blue L_{(x,y,b)}}  \ar@{-}[d] & &  {\blue L_{(x,a,y)}}  \ar@{-}[d] \\
{\red  L_{(a,x,y)}}  \ar@{-}[rd] & &  {\red L_{(b,y,x)}} \ar@{-}[ld] \\
& \blue{ L_{(b,y,a)}} & 
}
\end{equation*}
The same definitions and conventions can be given for $U$, and for any other quintary map
$\cX^5 \to \cX$ that is invariant under the Klein $4$-group acting on $(x,a,b,z)$.
In other words, the ``hexad'' (\ref{eqn:hexad}) comes from $\bullet$ via the canonical action of the group
 $\fS_3 =  \fS_4 / V$ acting on the $4$ variables $(v_1,v_2,v_3,v_4)=(x,a,b,z)$
(Diagram (\ref{eqn:hexad2})).
Following a terminology from the theory of loops and quasigroups, 
for a fixed triple $(a,y,b)$, we call the
six products in the ``hexad'' given above the {\em parastrophes}
of the ``principal'' product $\bullet$. 
 \end{definition}

\begin{theorem}\label{th:associativity}
Let $\cX$ be a bounded distributive lattice, and fix a triple of elements
in $\cX$. Then any of the six products by the ``hexad'' given above
is associative. Products belonging to opposite vertices of the
hexad are opposite in the algebraic sense ($u \cdot^{\op} v =
v \cdot u$). In particular, products of type $L_{a,y,a}$ or
$L_{a,a,y}$ are  commutative. 
Moreover, all semigroups $(\cX,\cdot)$ thus obtained are
{\em weak bands}, i.e., left and right multiplications are idempotent:
$$
\forall u,v \in \cX : \qquad u(uv)=uv=(uv)v.
$$
\end{theorem}

\begin{proof}
By the 
\href{https://en.wikipedia.org/wiki/Distributive_lattice#Representation_theory}{\em representation theorem for bounded distributive lattices}
(\cite{Stone}, cf.\ \cite{Bi}, Chapter IX, Theorem 11), 
we can imbed $\cX$ as a sublattice into the lattice of subsets
$\cP(M)$ of a set $M$.
Associativity (and the weak band property) are algebraic identities
 for $L$ in the sense of universal
algebra; if they hold in the lattice $\cP(M)$, then they hold also in the
lattice $\cX$.
Thus we shall assume in the sequel that $\cX = \cP(M)$, and 
we prove the claim in this case.
The advantage of $\cP(M)$ is that it has a {\em complementation},
which makes  it a {\em Boolean lattice}.

\begin{definition}\label{def:M} 
Assume $\cX = \cP(M)$, and denote for $u \in \cX$ by
 $\lnot u := M \setminus u$  its complement.
 Fix a triple $(a,y,b) \in \cX^3$. With respect to this datum, we decompose 
$$
M = M^\true \sqcup M^\false  \sqcup M^\rechts \sqcup M^\links \sqcup M^\und \sqcup M^\oder
$$
 into  a disjoint union of  six subsets, as follows:
\begin{enumerate}
\item
$\qquad M^\true := a \land y \land b$
\item
$\qquad M^\false := \lnot a \land \lnot y \land \lnot b$
\item
$\qquad M^\rechts := b \setminus a = b \land  (\lnot a)$
\item
$\qquad M^\links := a \setminus b = a \land (\lnot b)$
\item
$\qquad M^\oder := (a \land b) \setminus y = a \land b \land (\lnot y)$
\item
$\qquad M^\und := y \setminus (a \lor b) = y \land (\lnot a) \land  (\lnot b)$ 
\end{enumerate}
In other words, an element $\omega \in M$ belongs to these sets iff 
the triple of propositions $(\omega \in a, \omega \in y, \omega \in b)$
has the  following Boolean values:
\begin{enumerate}
\item
$\qquad (1,1,1)$
\item
$\qquad  (0,0,0)$
\item
$\qquad  (0,0,1)$ or $(0,1,1)$
\item
$\qquad  (1,0,0)$ or $(1,1,0)$
\item
$\qquad (1,0,1)$
\item
$\qquad  (0,1,0)$
\end{enumerate}
\end{definition}

\begin{remark}
The following arguments not only establish associativity of
 the ``principal'' product  (\ref{eqn:principal}), but also give
 an independent proof of the identity $L=U$ in case of a 
 power set lattice $\cX = \cP(M)$.
\end{remark}

\nin
Now let $(x,a,y,b,z) \in \cX^5$ and $\omega \in M$.
Then $\omega$ may belong, or not, to some of the $5$ compenents
of this quintuplet --
there are $2^5 = 32$ possible cases.  In the
following tables, we write 
 $0$ if $\omega$ does not belong to the set, and
 $1$ if $\omega$ belongs to the set. 
These $32$ cases can be partitioned according to the $6$ cases
defined above, as follows. 
For further information, we list also the Boolean values for 
$\omega$ belonging to elements of the lists
 ${\tt L} = (L_1,L_2,L_3,L_4)$ and 
 ${\tt U} = (U_1,U_2,U_3,U_4)$, as well as for $L_5$ and $U_5$.
 
 \msk
\nin
(1) Let $\omega \in M^\true$. 
The Boolean values for $\omega \in$ (or $\notin$)
$L,U,L_1,U_1,L_4,U_4$ all coincide.

\msk
\begin{center}
\begin{tabular}{l | l | l | l |  l     |   l | l  |   l        |    l       l  l             l  |   l}  
$x$   &  $a$   &  $y$ &  $b$  & $z$ &  list {\tt L} & $\lor {\tt L} = \land {\tt U}$
&  list {\tt U} & $L_5$ & $U_5$  \cr
\hline
$1$ & $1$ & $1$ & $1$ &$1$ &
$(1,1,1,1)$ & $\qquad 1$  & $(1,1,1,1)$ &
$1$ & 1
\\
$1$ & $1$ & $1$ & $1$ &$0$ &
$(1,0,1,1)$ & $\qquad 1$  & $(1,1,1,1)$ &
$1$ &  1
\cr 
$0$ & $1$ & $1$ & $1$ &$1$ &
$(1,1,0,1)$ & $\qquad 1$  & $(1,1,1,1)$ &
$1$ &  1
\cr 
$0$ & $1$ & $1$ & $1$ &$0$ &
$(1,0,0,1)$ & $\qquad 1$  & $(1,1,1,1)$ &
$ 0$ &  1 
\cr 
 \end{tabular}
\end{center}
\msk

\nin (2)  Let $\omega \in M^\false$.  
The Boolean values for $\omega \in$ (or $\notin$)
$L,U,L_1,U_1,L_4,U_4$ coincide.

\msk
\begin{center}
\begin{tabular}{l | l | l | l |  l     |   l | l  |   l        |    l       l  l             l  |   l}  
$x$   &  $a$   &  $y$ &  $b$  & $z$ &  list {\tt L} & $\lor {\tt L} = \land {\tt U}$
&  list {\tt U} & $L_5$ & $U_5$  \cr
\hline
$0$ & $0$ & $0$ & $0$ &$0$ &
$(0,0,0,0)$ & $\qquad 0$  & $(0,0,0,0)$ &
$0$ & $0$ 
\\
$0$ & $0$ & $0$ & $0$ &$1$ &
$(0,0,0,0)$ & $\qquad 0$  & $(0,0,1,0)$ &
$0$ &  0
\\
$1$ & $0$ & $0$ & $0$ &$0$ &
$(0,0,0,0)$ & $\qquad 0$  & $(0,1,0,0)$ &
$0$ &  0
\cr 
$1$ & $0$ & $0$ & $0$ &$1$ &
$(0,0,0,0)$ & $\qquad  0$  & $(0,1,1,0)$ & 
$ 0$ &  1
\cr 
\end{tabular}
\end{center}
\msk

\nin (3) Let $\omega \in M^\rechts$.  
The Boolean values for $\omega \in$ (or $\notin$)
 $L,U,L_1,U_1,L_5,U_5$ coincide.

\msk
\begin{center}
\begin{tabular}{l | l | l | l |  l     |   l | l  |   l        |    l       l  l             l  |   l}  
$x$   &  $a$   &  $y$ &  $b$  & $z$ &  list {\tt L} & $\lor {\tt L} = \land {\tt U}$
&  list {\tt U} & $L_5$ & $U_5$  \cr
\hline
$0$ & $0$ & $0$ & $1$ &$0$ &
$(0,0,0,0)$ & $\qquad 0$  & $(0,0,0,1)$ &
$0$ &  0
\\
0 & 0 & 0 & 1 & 1 & $(1,1,0,0)$ & $\qquad 1$ & $(1,1,1,1)$ & 1 & 1
\cr 
$1$ & $0$ & $0$ & $1$ &$0$ &
$(0,0,0,0)$ & $\qquad 0$  & $(0,1,0,1)$ &
$0$ &  0
\cr 
$1$ & $0$ & $0$ & $1$ &$1$ &
$(1,1,0,0)$ & $\qquad 1$  & $(1,1,1,1)$ &
$1$ & 1
\cr 
\hline
$0$ & $0$ & $1$ & $1$ &$0$ &
$(0,0,0,0)$ & $\qquad 0$  & $(0,1,0,1)$ &
$0$ &  0
\\
0 & 0 & 1 & 1 & 1 & $(1,1,0,0)$ & $\qquad 1$ & $(1,1,1,1)$ & 1 & 1
\cr 
$1$ & $0$ & $1$ & $1$ &$0$ &
$(0,0,0,0)$ & $\qquad 0$  & $(0,1,0,1)$ &
$0$ &  0
\cr 
$1$ & $0$ & $1$ & $1$ &$1$ &
$(1,1,0,0)$ & $\qquad 1$  & $(1,1,1,1)$ &
$1$ & 1
\cr 
 \end{tabular}
\end{center}
\msk
 
\nin (4) Let $\omega \in M^\links$.  The Boolean values for $\omega \in$ (or $\notin$)
$L ,U, L_4, U_4,L_5,U_5$ coincide. 

\msk
\begin{center}
\begin{tabular}{l | l | l | l |  l     |   l | l  |   l        |    l       l  l             l  |   l}  
$x$   &  $a$   &  $y$ &  $b$  & $z$ &  list {\tt L} & $\lor {\tt L} = \land {\tt U}$
&  list {\tt U} & $L_5$ & $U_5$  \cr
\hline
$0$ & $1$ & $0$ & $0$ &$0$ &
$(0,0,0,0)$ & $\qquad 0$  & $(1,0,0,0)$ &
$0$ &  0
\\
$0$ & $1$ & $0$ & $0$ &$1$ &
$(0,0,0,0)$ & $\qquad 0$  & $(1,0,1,0)$ &
$0$ &  0
\\
$1$ & $1$ & $0$ & $0$ &$0$ &
$(0,0,1,1)$ & $\qquad 1$  & $(1,1,1,1)$ &
$1$ &  1
\\
$1$ & $1$ & $0$ & $0$ &$1$ &
$(0,0,1,1)$ & $\qquad 1$  & $(1,1,1,1)$ &
$1$ &  1
\\
\hline
$0$ & $1$ & $1$ & $0$ &$0$ &
$(0,0,0,0)$ & $\qquad 0$  & $(1,0,0,0)$ &
$0$ &  0
\\
$0$ & $1$ & $1$ & $0$ &$1$ &
$(0,0,0,0)$ & $\qquad 0$  & $(1,0,1,0)$ &
$0$ &  0
\\
$1$ & $1$ & $1$ & $0$ &$0$ &
$(0,0,1,1)$ & $\qquad 1$  & $(1,1,1,1)$ &
$1$ &  1
\\
$1$ & $1$ & $1$ & $0$ &$1$ &
$(0,1,1,1)$ & $\qquad 1$  & $(1,1,1,1)$ &
$1$ &  1
\\
 \end{tabular}
\end{center}
\msk

\nin (5) Let $\omega \in M^\oder$. The Boolean values for $\omega \in$ (or $\notin$)
 $L,U,L_5, U_2,U_3$ coincide.

\msk
\begin{center}
\begin{tabular}{l | l | l | l |  l     |   l | l  |   l        |    l       l  l             l  |   l}  
$x$   &  $a$   &  $y$ &  $b$  & $z$ &  list {\tt L} & $\lor {\tt L} = \land {\tt U}$
&  list {\tt U} & $L_5$ & $U_5$  \cr
\hline
$0$ & $1$ & $0$ & $1$ &$0$ &
$(0,0,0,0)$ & $\qquad 0$  & $(1,0,0,1)$ &
$0$ &  1
\cr 
$0$ & $1$ & $0$ & $1$ &$1$ &
$(1,1,0,0)$ & $\qquad 1$  & $(1,1,1,1)$ &
$1$ &  1
\cr 
$1$ & $1$ & $0$ & $1$ &$0$ &
$(0,0,1,1)$ & $\qquad 1$  & $(1,1,1,1)$ &
$1$ &  1
\cr 
$1$ & $1$ & $0$ & $1$ &$1$ &
$(1,1,1,1)$ & $\qquad 1$  & $(1,1,1,1)$ &
$1$ &  1
\cr 
 \end{tabular}
\end{center}
\msk

\nin (6)  Let $\omega \in M^\und$. 
The Boolean values for $\omega \in$ (or $\notin$)
 $L,U,U_5,L_2,L_3$  coincide.
 
\msk
\begin{center}
\begin{tabular}{l | l | l | l |  l     |   l | l  |   l        |    l       l  l             l  |   l}  
$x$   &  $a$   &  $y$ &  $b$  & $z$ &  list {\tt L} & $\lor {\tt L} = \land {\tt U}$
&  list {\tt U} & $L_5$ & $U_5$  \cr
\hline
0 & 0 & 1 & 0 & 0 & $(0,0,0,0)$ & $\qquad 0$ & $(0,0,0,0)$ & 0 & 0 
\\
0 & 0 & 1 & 0 & 1 & $(0,0,0,0)$ & $\qquad 0$ & $(1,0,1,0)$ & 0 & 0 
\\
1 & 0 & 1 & 0 & 0 & $(0,0,0,0)$ & $\qquad 0$ & $(0,1,0,1)$ & 0 & 0 
\\
1 & 0 & 1 & 0 & 1 & $(0,1,1,0)$ & $\qquad 1$ & $(1,1,1,1)$ & 0 & 1
\\
 \end{tabular}
\end{center}

\msk
\begin{remark}\label{rk:strange}
Simple inspection of the tables shows
that the Boolean values for $L$ and $U$ coincide in all possible 
cases, whence a proof of $L=U$ in the power set case. 
Moreover, one may observe that
in all but one cases, $L$ coincides with $L_{23}:=L_2 \lor L_3$, resp.\ with
$L_{14}:=L_1 \lor L_4$, and in all but one cases, $U$ coincides with
$U_{14}:=U_1 \land U_4$, resp.\ with $U_{23}:=U_2 \land U_3$.
In all but two cases, $L$ coincides with $L_5$. Likewise for
$U$ and $U_5$. 
Also, the rule of defining $L$ and $U$ is rather close to a 
\href{https://en.wikipedia.org/wiki/Majority_function}{\em majority function}
(``the winner takes it all'': if $\omega$ belongs to at least $3$ of the $5$
sets, then it belongs to $L$ and to $U$) -- however, 
$4$ among the $32$ cases  violate this rule.
\end{remark}

Coming back to the proof of associativity of 
the product
$(x,z) \mapsto x\bullet z =  x \cdot_{a,y,b} z $, 
direct inspection of the tables shows that, when restricting $\bullet$
to the six subsets from Def.\ \ref{def:M} , we get in the respective cases:
\begin{enumerate}
\item
on $M^\true$, $\bullet$ is the
``true'' connector $(x,z) \mapsto 1_{M^\true}$,
\item
on $M^\false$, it is  the ``false'' connector
$(x,z) \mapsto 0_{M^\false}$,
\item
on $M^\rechts$, it is the ``second'' connector
$(x,z) \mapsto z$,
\item
on $M^\links$, it is the ``first'' connector
$(x,z) \mapsto x$,
\item
on $M^\oder$, it is  the ``or'' connector
$(x,z) \mapsto x \lor z$,
\item 
on $M^\und$, it is the ``and'' connector
$(x,z) \mapsto x \land z$. 
\end{enumerate} 
Now, each of the  six binary operations listed above obviously
 is associative.\footnote{Remarkably, these six
products correspond exactly to the 
\href{https://en.wikipedia.org/wiki/Semigroup_with_two_elements}{six semigroup laws on the set of two elements which are not group laws},
and also to the 
\href{https://en.wikipedia.org/wiki/Distributive_lattice}{free distributive lattice on 2 generators}.}
From this it follows that $\bullet$ is also associative:
for
$i \in I:=  \{ \true, \false, \rechts, \links, \oder, \und \}$, and $x \in \cP(M)$,
 let
$x^i := x \cap M^i$, so $x = \sqcup_{i\in I} x^i$ (disjoint union).
Then, by what we have just seen,
$(x \bullet z)^i = x^i\bullet  z^i$ is given by the connector 
corresponding to $i$, hence
\begin{align*}
(x \bullet w) \bullet z & = \sqcup_{i\in I} ( (x \bullet w) \bullet z)^i =
\sqcup_{i\in I}  (x^i \bullet w^i) \bullet z^i
\\
& =\sqcup_{i\in I}  x^i \bullet (w^i \bullet z^i) =
x \bullet (w \bullet z).
\end{align*}
Moreover, each of the six connectors has the ``weak band property'',
as is immediately checked, and hence by the same argument,
$\bullet$ also is a weak band. 
Notice also that the opposite product of $\bullet$ is gotten by exchanging
$a$ and $b$ (which corresponds to exchanging 
the indices ``right'' and ``left'' and keeping the other four), which
corresponds to the opposite vertex in the ``hexad''.

\msk
Next, consider  the product $L_{(x,a,y)}$.  We 
present the information contained in the precding tables, 
structured in a different way:
fix the Boolean values for $(\omega \in x,\omega\in a,\omega \in y)$ and 
consider those for $\omega \in L(x,a,y,b,z)$ as a function of those of
$(\omega \in b,\omega \in z)$. 

\msk 
1.  When $\omega$ is in at most one of the sets
$x,a,y$, then we get the ``and''-connector:

\begin{multicols}{2}

\begin{tabular}{l | l | l | l |  l     |   l } 
$x$   &  $a$   &  $y$ &  $b$  & $z$ &  $L=U$
  \cr
\hline
$0$ & $0$ & $0$ & $0$ &$0$ &
 $\quad 0$  
\\
$0$ & $0$ & $0$ & $0$ &$1$ &
 $\quad 0$  
\cr 
$0$ & $0$ & $0$ & $1$ &$0$ &
 $\quad 0$   
\cr 
$0$ & $0$ & $0$ & $1$ &$1$ &
 $\quad 1$   
\cr 
 \end{tabular}

\columnbreak

\begin{tabular}{l | l | l | l |  l     |   l } 
$x$   &  $a$   &  $y$ &  $b$  & $z$ &  $L=U$
  \cr
\hline
$0$ & $0$ & $1$ & $0$ &$0$ &
 $\quad 0$  
\\
$0$ & $0$ & $1$ & $0$ &$1$ &
 $\quad 0$  
\cr 
$0$ & $0$ & $1$ & $1$ &$0$ &
 $\quad 0$   
\cr 
$0$ & $0$ & $1$ & $1$ &$1$ &
 $\quad 1$   
\cr 
 \end{tabular}
\end{multicols}

\ssk

\begin{multicols}{2}

\begin{tabular}{l | l | l | l |  l     |   l }
$x$   &  $a$   &  $y$ &  $b$  & $z$ &  $L=U$
  \cr
\hline
$0$ & $1$ & $0$ & $0$ &$0$ &
 $\quad 0$  
\\
$0$ & $1$ & $0$ & $0$ &$1$ &
 $\quad 0$  
\cr 
$0$ & $1$ & $0$ & $1$ &$0$ &
 $\quad 0$   
\cr 
$0$ & $1$ & $0$ & $1$ &$1$ &
 $\quad 1$   
\cr 
 \end{tabular}

\columnbreak

\begin{tabular}{l | l | l | l |  l     |   l }
$x$   &  $a$   &  $y$ &  $b$  & $z$ &  $L=U$
  \cr
\hline
$1$ & $0$ & $0$ & $0$ &$0$ &
 $\quad 0$  
\\
$1$ & $0$ & $0$ & $0$ &$1$ &
 $\quad 0$  
\cr 
$1$ & $0$ & $0$ & $1$ &$0$ &
 $\quad 0$   
\cr 
$1$ & $0$ & $0$ & $1$ &$1$ &
 $\quad 1$   
\cr 
 \end{tabular}
\end{multicols}
\ssk

\nin 2.   When $\omega \notin x$ and
$\omega \in a \land y$, then we get the
``left''-connector:

\begin{center}
\begin{tabular}{l | l | l | l |  l     |   l }
$x$   &  $a$   &  $y$ &  $b$  & $z$ &  $L=U$
  \cr
\hline
$0$ & $1$ & $1$ & $0$ &$0$ &
 $\quad 0$  
\\
$0$ & $1$ & $1$ & $0$ &$1$ &
 $\quad 0$  
\cr 
$0$ & $1$ & $1$ & $1$ &$0$ &
 $\quad 1$   
\cr 
$0$ & $1$ & $1$ & $1$ &$1$ &
 $\quad 1$   
\cr 
 \end{tabular}
\end{center}
\ssk

\nin 3.  When $\omega \notin a$ and
$\omega \in x \land y$, then we get the
``right''-connector:

\begin{center}
\begin{tabular}{l | l | l | l |  l     |   l }
$x$   &  $a$   &  $y$ &  $b$  & $z$ &  $L=U$
  \cr
\hline
$1$ & $0$ & $1$ & $0$ &$0$ &
 $\quad 0$  
\\
$1$ & $0$ & $1$ & $0$ &$1$ &
 $\quad 1$  
\cr 
$1$ & $0$ & $1$ & $1$ &$0$ &
 $\quad 0$   
\cr 
$1$ & $0$ & $1$ & $1$ &$1$ &
 $\quad 1$   
\cr 
 \end{tabular}
\end{center}

\nin 4.  When $\omega \in x \land a$, then we get the  ``true'' connector:

\begin{multicols}{2}
\begin{tabular}{l | l | l | l |  l     |   l }
$x$   &  $a$   &  $y$ &  $b$  & $z$ &  $L=U$
  \cr
\hline
$1$ & $1$ & $0$ & $0$ &$0$ &
 $\quad 1$  
\\
$1$ & $1$ & $0$ & $0$ &$1$ &
 $\quad 1$  
\cr 
$1$ & $1$ & $0$ & $1$ &$0$ &
 $\quad 1$   
\cr 
$1$ & $1$ & $0$ & $1$ &$1$ &
 $\quad 1$   
\cr 
 \end{tabular}

\columnbreak

\begin{tabular}{l | l | l | l |  l     |   l } 
$x$   &  $a$   &  $y$ &  $b$  & $z$ &  $L=U$
  \cr
\hline
$1$ & $1$ & $1$ & $0$ &$0$ &
 $\quad 1$  
\\
$1$ & $1$ & $1$ & $0$ &$1$ &
 $\quad 1$  
\cr 
$1$ & $1$ & $1$ & $1$ &$0$ &
 $\quad 1$   
\cr 
$1$ & $1$ & $1$ & $1$ &$1$ &
 $\quad 1$   
\cr 
 \end{tabular}
\end{multicols}
\msk

\nin
In all four cases, we get a semigroup law satisfying the 
``weak band'' property, and decomposing $M$ into a
disjoint union of sets where these properties hold, it follows as above
that the product $L_{(x,a,y)}$ has these properties on $M$. 
The product $L_{(a,x,y)}$ is the opposite product of $L_{(x,a,y)}$,
 and hence also has
the same properties -- note that exchange of $(a,x)$ just exchanges
cases 2. and 3.  above. 
Next, fix the Boolean values for $\omega \in x,y,b$ and list
those of $\omega \in L(x,a,y,b,z)$ as function of 
those of $\omega \in (a,z)$.
The result is similar as above, in a ``dual'' way: 

\ssk
(1) When $\omega$ is in at most one of 
$\lnot x,\lnot b,\lnot y$, then we get the ``or''-connector.

\ssk
(2) 
When $\omega \in x$ and
$\omega \in (\lnot b \land \lnot y)$, then we get the
``left''-connector.

\ssk
(3)
When $\omega \in b$ and
$\omega \in (\lnot x \land \lnot y)$, then we get the
``right''-connector.

\ssk
(4)
When $\omega \in (\lnot x \land \lnot b)$, then we get
the ``false'' connector.

\ssk
\nin 
As above, this implies the statements of Theorem \ref{th:associativity} for
$L_{(x,y,b)}$ and $L_{(b,y,x)}$, and finishes its proof.
\end{proof}

\begin{definition}
For each of the products  from the ``hexad'', we call its 
{\em bottom} the set corresponding to the ``true'' connector,
and we shall the use the notation $M^\true$ already used for $\bullet$.
We call its {\em top} the set corresponding to the ``false''  connector,
and we shall use the notation $M^\false$ already used for $\bullet$. 
Thus, for all $u,v \in \cX$,
$$
M^\true \leq u \cdot v \leq (M \setminus  M^\false) = \lnot M^\false .
$$
\end{definition}

\begin{theorem}\label{th:bottom-top}
Let   $\cX = \cP(M)$ be the lattice of the power set of a set $M$.
For the six products from the hexad, bottom and top are given by
\begin{enumerate}
\item
for $\bullet = L_{(a,y,b)}$, we have
$M^\true = a \land y \land b$ and $M^\false = \lnot (a \lor y \lor b)$,
\item
for  $L_{(x,a,y)}$, we have  $M^\true = x \land a$,
and $M^\false = \emptyset$,
\item
for  $L_{(x,y,b)}$, we have 
$M^\false =  (\lnot x \land \lnot b) = \lnot (x \lor b)$,
and $M^\true = \emptyset$,
\item
products belonging to opposite hexad vertices have same
bottom and top.
\end{enumerate}
\end{theorem}

\begin{proof}
All statements follow by direct inspection from the tables given above.
\end{proof}

\begin{remark}
We'll see in the arithmetic case that the existence of a non-empty
bottom causes a certain {\em square periodicity} of the respective products. 
More precisely, there may also be a {\em periodicity in a single
argument}, related to the following result.
\end{remark}

\begin{theorem}\label{th:period}
Let   $\cX = \cP(M)$ be the lattice of the power set of a set $M$.
Fix a triple $(a,y,b) \in \cX^3$. Then the principal product 
 $x \bullet z= x\cdot_{a,y,b} z$
is 
\begin{enumerate}
\item
{\em $a \land y$-periodic} in $x$, i.e., for all $x$ and $z$:
$\quad (x \lor (a\land y)) \bullet  z = x \bullet  z $,
\item
{\em $b \land y$-periodic in $z$}, i.e., for all $x$ and $z$:
$\quad  x \bullet  z = x \bullet (z \lor (b \land y))$.
\end{enumerate}
\end{theorem}

\begin{proof}
Claim (1) amounts so saying that 
$L'=L(x \lor (a \land y),a,y,b,z)$ and
$L=L(x,a,y,b,z)$ have same Boolean values for $\omega \in M$.
Now, looking at the tables of Boolean values 
describing $\bullet$, the cases $\omega \in a \land y$ correspond
to $M^\true$ (where $L'$ and $L$ both always have value $1$),
and $M^\links$ (containing two lines which are distinguished only
by one entry, corresponding to the periodicity claim).
Likewise for Claim (2).
\end{proof}

\nin
\textbf{Ternary products.} --
In certain contexts it is useful to switch the viewpoint from {\em binary}
to {\em ternary} products, e.g., when looking at the ``affine analog'' of
groups, which have been called {\em torsor} in \cite{BeKi1}.
Fix a pair
$(a,b) \in \cX^2$, and consider the remaining {\em three} arguments
$(x,y,z)$ as variables, i.e., we define a {\em ternary product}
\begin{equation}\label{eqn:ternary}
\cX^3 \to \cX, \quad 
(x,y,z) \mapsto (xyz)_{ab}:=L(x,a,y,b,z) = x\cdot_{a,y,b} z.
\end{equation}

\begin{theorem}
Let   $\cX$ be a bounded distributive lattice, and 
fix $(a,b) \in \cX^2$.  Then: 
\begin{enumerate}
\item
the product {\rm (\ref{eqn:ternary})} is
{\em associative} and
{\em para-associative}: 
$\forall x,x',y,z,z' \in \cX$,
\begin{align*}
((x x' y)_{ab} z'z)_{ab} & =
(x (x' y z')_{ab}z)_{ab} 
\\
&=
(x (z' y x')_{ab}z)_{ab} =
(xx' (yz'z)_{ab})_{ab}  ,
\end{align*}
\item
the ``middle'' operator $M_{x,z} ( y):= (xyz)_{ab}$ is idempotent;
more precisely,
$$
M_{x,z}\circ M_{x,z} = M_{x,z} = M_{x,z}\circ M_{z,x} .
$$
\end{enumerate}
\end{theorem}

\begin{proof} (1)
As for the binary products, 
it suffices to prove the theorem for $\cX = \cP(M)$, a 
power set lattice. 
With respect to the fixed pair $(a,b)$, we
decompose $M$ into the disjoint union of four subsets, using
notation from Def.\ \ref{def:M},
$$
M = \bigl(M^\true \sqcup M^\und \bigr) \sqcup \bigl(M^\false \sqcup M^\oder \bigr)
\sqcup M^\rechts \sqcup M^\links .
$$
Then, as is seen directly from the tables describing the Boolean values
of $\bullet$, 
the ternary map
$(x,y,z) \mapsto (xyz)_{ab}:= L(x,a,y,b,z)$ coincides with

\ssk
$(x,y,z) \mapsto x \land y \land z$ on the first set $M^\true \sqcup M^\und$,

$(x,y,z) \mapsto x \lor y \lor z$ on the second set $M^\false \sqcup M^\oder$,

$(x,y,z) \mapsto z$ on the third set $M^\rechts$ ,

$(x,y,z) \mapsto x$ on the fourth set $M^\links$.

\ssk
\nin
Thus the ternary product $(xyz)_{ab}$ will inherit properties that
are in common for these four elementary ones.
Now, it is immediately checked that all four ``elementary'' ternary products
are associative and para-associative, whence (1).
Concerning (2),  it is clear that in all four cases 
 middle multiplication operators are idempotent.
Eg., for the right product $(xyz)=z$, we get 
$(x(xyz)z) = (xzz) = z = (xyz)$ and
$(x(zyx)z)= (xxz)=z = (xyz)$. 
 \end{proof}

\begin{remark}
There is no  ``hexad'' of ternary products:
the preceding result does not carry over to the other
``parastrophe''  ternary products.  
\end{remark}

\section{The arithmetic case}\label{sec:arithmetic}

\subsection{Comparison with the Boolean case, and general results}
Now consider the possibly most interesting distributive lattice,
the lattice $\cX = \N_0$ of natural numbers with
$\land = \lcm$ and $\lor = \gcd$.
Via the {\em $p$-adic valuation} $n \mapsto v_p(n)$, for each
prime number $p$, transforming the lattice operations to min and
max,
\begin{equation}
\begin{matrix}
v_p(\lcm(x,y)) = \max (v_p(x),v_p(y)),\\
v_p(\gcd(x,y)) = \min (v_p(x),v_p(y)).
\end{matrix}
\end{equation}
it is closely related to the {\em totally ordered case}
(Theorem \ref{th:order}).
This imbedding can also be formulated as an imbedding into
a power set $\cP(\PP)$ (Remark \ref{rk:second_proof}), 
realizing the arithmetic case as a sublattice of a Boolean lattice. 
However, 
 the arithmetic lattice  seems to be ``quite far away'' from
the Boolean case, and has a rather particular flavor.

\begin{remark}\label{rk:second_proof}
The imbedding of the lattice
$(\N_0,\land,\lor)$ into
a power set lattice can be described  directly as follows: let 
\begin{equation}
\bP := \{ p^k \mid \, p \mbox{ prime, or } p=1, \mbox{ and } k \in \N \} \subset \N,
\end{equation}
the set of
prime powers.
 Assembling  the evaluation maps into a single object, let
\begin{equation}
\lambda : \N_0 \to \cP(\bP), \quad
0 \mapsto \bP, \,
n \mapsto 
\lambda(n) := \{ m \in \bP \mid \,  \, m \big\vert n \} ,
\end{equation}
sending $n$ to
the set of prime power divisors of $n$. Note that, with our choice of
notation, this version of $\lambda$ is {\em antitone} (this is inevitable if 
we want to denote both intersection of ideals in $\Z$ and of subsets of $\bP$ by
the same symbol $\land$):
\begin{align*}
\lambda (n \land m) =
\lambda (\lcm (n ,m)) &= \lambda(n) \cup \lambda(m) =
\lambda(n) \lor  \lambda(m) , 
\\
\lambda(n \lor m) =
\lambda (\gcd( n , m)) &= \lambda(n) \cap \lambda(m) = 
\lambda(n) \land \lambda(m).
\end{align*}
Using this imbedding, we can complete $\N_0$ to a Boolean algebra, 
adding the missing complement map, simply by taking the sublattice 
$\cX(\N_0) \subset \cP(\bP)$ generated by $\lambda(\N_0)$ and
all complement sets $\overline n := \lnot (\lambda(n))$ for $n \in \N_0$.
It is quite easy to see that $\cX(\N_0) = \cP^{\rm fin}(\bP) \sqcup \cP^{\rm cof}(\bP)$
is precisely the 
\href{https://en.wikipedia.org/wiki/Cofiniteness#Boolean_algebras}{\em finite-cofinite algebra of $\bP$}, 
i.e., the union of the collection of all {\em finite} subsets of $\bP$  and all
{\em cofinite} subsets of $\bP$. Being a Boolean algebra, it is stable under 
$U$ and $L$. 
\end{remark}


\begin{definition}
 For a fixed triple
 $(a,y,b) \in \N^3$, we call 

\ssk
$n = n_{y,b}:=  \lcm (y,b)$ ``line period'',

$m := m_{y,a} := \lcm(y,a)$ ``column period'',

$N := N_{a,y,b} := \lcm(a,y,b)$ ``square period'' (bottom),

$K := K_{a,y,b}: = \gcd(a,y,b)$ ``base frequency'' (top).
\end{definition}

\nin
We shall mostly be interested in the ``principal products''  $\bullet_{a,y,b}$.
When $N = 0$, then its description follows
directly from Theorem \ref{th:diagonals}:

\begin{theorem}
Let $(a,y,b) \in \N_0^3$ such that $N = 0$. Then,
\begin{enumerate}
\item
if $y=0$, we have $x \bullet_{a,0,b} z = L_5(x,a,y,b,z) = (b \land z) \lor (a \land x) $,
\item
if $a=0$, we have $x \bullet_{0,y,b} z = L_2(x,a,y,b,z) = z \land (b \lor (x \land y)) $,
\item
if $b=0$, we have $x \bullet_{a,y,0} z = L_3 (x,a,y,b,z) = x \land (a \lor (z\land y))$.
\end{enumerate}
In particular, we have 
\begin{enumerate}
\item
$ x \bullet_{1,0,1} z  = x \lor z$,
\item
$x \bullet_{0,1,0} z  = x \land  z $,
\item
$x \bullet_{1,0,0} z  = x = x \bullet_{1,1,0} z = z \bullet_{0,1,1} x = z \bullet_{0,0,1} x$ .
\end{enumerate}
\end{theorem}

\begin{definition}
Let $(a,y,b) \in \N^3$ (so $N \not=0$ and $K \not= 0$),
let $[1,KN]$ be the set of divisors of $KN$. For $x \in [1,KN]$,  we call
 $x' := \frac{KN}{x}$ its {\em conjugate divisor}, and define the {\em conjugate product}
$$
\N_0^2 \to \N_0, \quad (x,z) \mapsto x \bullet_{b',y',a'} z .
$$
\end{definition}

\begin{theorem}\label{th:arithmetic_case}
Let $(a,y,b) \in \N_0^3$ such that $N \not= 0$.
Then:
\begin{enumerate}
\item
For all $(x,z) \in \N_0^2$, 
 $L(x,a,y,b,z)$ belongs to the set of divisors of
$N$, and is a multiple of $K$. In particular,
$x \bullet_{a,y,b} z$  takes only a finite number of values.
\item 
The product $\bullet_{a,y,b}$ is {\em $(n,m)$-periodic} in the following sense:

\ssk
$
x \equiv x'  {\rm mod} (n)  \quad \Rightarrow \quad
x \bullet_{a,y,b} z= x' \bullet_{a,y,b} z,
$

$
 z \equiv z'  {\rm mod}(m) \quad \Rightarrow \quad
 x \bullet_{a,y,b} z = x \bullet_{a,y,b} z' .
$

\ssk
\nin
It follows that
the product $\bullet_{a,y,b}$ passes to the quotient defining a product on
$\Z / N\Z $, or even on $K\Z / N \Z$.
\item
The ``conjugation map''  
$$
\gamma : [1,KN] \to [1,KN], \quad x \mapsto x' = \frac{KN}{x} 
$$
is an isomorphism from the semigroup $[1,KN]$ with product
$\bullet_{a,y,b}$ onto the semigrop $[1,KN]$ with conjugate product
$\bullet_{b',y',b'}$. 
By restriction, the same holds for the subsemigroup
$[K,N]=\{ d \mid \, K\vert d, \,\,  d \vert N \}$.
\end{enumerate}
\end{theorem}

\begin{proof}
(1) and (2) are a reformulation of parts of
Theorems \ref{th:associativity}, \ref{th:bottom-top}, and \ref{th:period}.

(3) Clearly, $\gamma$ is a lattice involution (antiautomorphism of order $2$) of the
lattice $[1,KN]$ (cf.\  Example \ref{ex:anti-auto}).
The claim now follows from Lemma \ref{la:anti}. 
\end{proof}

\begin{remark}
The isomorphism from item (3) does not directly extend to an isomorphism
$(\N_0, \bullet_{a,y,b})\to (\N_0,\bullet_{a',y',b'})$.
Rather than ``isomorphic'', these products are ``isotopic'', in the sense that
when imbedding $\N_0$ into a Boolean algebra (Remark \ref{rk:second_proof}), 
then both products are the ``same'', but evaluated on ``opposite'' copies of $\N_0$.
\end{remark}

\subsection{Examples of tables}

\begin{remark}\label{rk:Sage}
The reader may enjoy to do
 some numerical tests herself. 
 Here is the program which we have used for computing $L$ and $U$
 in the following tables, 
along with $L_i,U_i$, for $i=1,\ldots,4$,  by using
\href{http://www.sagemath.org}{\tt SageMath}, along with a random example:

\msk
{\tt
def Bounds(x,a,y,b,z):

    L=[lcm(b,gcd(z,lcm(a,y))), lcm(z,gcd(b,lcm(x,y))), 
    
    
   $\quad $  lcm(x,gcd(a,lcm(z,y))),
       lcm(a,gcd(x,lcm(b,y)))]  ; 
       
    U=[gcd(a,lcm(z,gcd(y,b))), gcd(x,lcm(b,gcd(z,y))),
    
    
    $\quad$ gcd(z,lcm(a,gcd(x,y))),
        gcd(b,lcm(x,gcd(a,y)))] ;
        
  print(L) ;  print(gcd(L)) ; 
  
  print(U) ; print(lcm(U)) 
 }

\begin{center}
\begin{tabular}{l  l     l   l      l        l    l     l }  
{\tt Bounds(425, 204, 1000, 200, 402)}
& & &
{\tt Bounds(425, 200, 1000,
204, 402)}
\\
{\tt [600, 40200, 5100, 5100] 
300}
&  & &
{\tt [204, 13668, 3400, 3400] 
68}
\\
{\tt [12, 25, 6, 100] 
300}
& & &
{\tt [4, 17, 2, 68] 
68}
\\
\end{tabular}
\end{center}
 \end{remark}

Now let us write up some ``multiplication tables'' for fixed $(a,y,b)$.
By Theorem \ref{th:diagonals},  the left upper corner
of all of the following multiplication tables (first column: values $x$, first line:
values z) 
is given by
\begin{align*}
0 \bullet_{a,y,b} 0 & = a \land y \land b = N  \mbox{ (the square period)}, 
\\
0 \bullet_{a,y,b} 1 & = b,  
\\
1 \bullet_{a,y,b} 0 & = a,
\\
1 \bullet_{a,y,b} 1 & = a \lor y \lor b = K  \mbox{ (the base frequency).} 
\end{align*} 
Also by Theorem \ref{th:diagonals},
when $a=b$, then $a$ always is an {\em absorbing element} (taking the role
of the zero element in usual quotient rings): $\forall x \in \N$,
$$
a\bullet_{a,y,a} x = L(a,a,y,a,x) = a = x \bullet_{a,y,a} a .
$$
In particular,  when $a = y = b$, then 
$x \bullet_{a,a,a} z= a$.

\begin{example}
The multiplication table for $x \bullet_{1y1} z = x \lor y \lor z$ is
$y$-periodic; if $y$ is a prime number, it takes two values; 
e.g., for  $y=2$:
\begin{center}
\begin{tabular}{l | l      l   l      l        l    l     l }  
$x\bullet_{1,2,1} z$ & 0 & 1 & 2 & 3 & 4 & 5 \\
\hline
0 & 2 & 1 & 2 & 1 & 2 & 1  \\
1 & 1 & 1 & 1 & 1 & 1 & 1 \\
2 & 2 & 1 & 2 & 1 & 2 & 1
\end{tabular}
\end{center}
When $y = 8$, the extracted table for the divisors of $N=8$ is a ``minimum-law'':
\begin{center}
\begin{tabular}{l | l      l   l      l        l    l     l }  
$x\bullet_{1,8,1} z$ & 1 & 2 & 4 & 8  \\
\hline
1 &  1 & 1 & 1 & 1  \\
2 & 1 & 2 & 2 & 2  \\
4 & 1 & 2 & 4 & 4 \\
8 & 1 & 2 & 4 & 8 
\end{tabular}
\end{center}
\end{example}

\begin{example}[Prime power cases]
These are the cases where
$p$ is a prime number and $(a,y,b)=(p^k,p^\ell,p^m)$ 
with $k,\ell,m \in \N_0$. For instance: 
\begin{center}
\begin{tabular}{l | l      l   l      l        l    l     l   l    l     l }  
$x\bullet_{8,4,2} z$ &0 & 1 & 2 & 3 & 4 & 5 & 6  & 7 & 8   \\
\hline
0 & 8 & 2 & 2 & 2 & 4 & 2 & 2 & 2 & 8 \\
1 & 8 & 2 & 2 & 2 & 4 & 2 & 2 & 2 & 8 \\
2 & 8 & 2 & 2 & 2 & 4 & 2 & 2 & 2 & 8 \\
\end{tabular}
\end{center}
One observes an effective line period $1$, instead of the predicted $n=4$.
Indeed, 
it follows from Theorem \ref{th:order} that
$x\cdot_{8,4,2} z = 2^{\max(1,\min(v_p(z),3))}$, which does not depend
on $x$.  In the prime power case,
$x \cdot_{a,y,b} z$ is solely determined by the $p$-adic valuation
$v_p(x), v_p(z)$ of $x$ and $z$ and the $L$- and $U$-functions 
for a totally ordered set (Theorem \ref{th:order}).
The cases $k \geq \ell \geq m$ and $k \leq \ell \leq m$ behave like the
preceding example. 
The other cases are more complicated.  For instance, here is the table
for the divisors of $16$ with $(a,y,b) = (8,1,16)$:
\begin{center}
\begin{tabular}{l | l      l   l      l        l    l     l }  
$x\bullet_{8,1,16} z$ & 1 & 2 & 4 & 8 & 16  \\
\hline
1 & 1 & 2 & 4 & 8 & 8  \\
2 & 2 & 2 & 4 & 8 & 8 \\
4 & 4 & 4 & 4 & 8 & 8 \\
8 & 8 & 8 & 8 & 8 & 8 \\
16 & 16 & 16 & 16 & 16 & 16 
 \end{tabular}
\end{center}
When $(a,y,b)=(8,1,8)$, the extracted table is similar: just skip the last line
and the last column of the preceding table.
Note that we get the table of a ``maximum-law'', which is  ``dual''
to the corresponding table of the conjugate product $(1,8,1)$.
\end{example}

\begin{example}
When $(a,y,b)=(p,1,q)$ with
$p,q$  two prime numbers, then the structure of the products
$x \bullet_{p,1,q} z = (p \lor x) \land (q \lor z)$ 
follows this pattern:
\begin{center}
\begin{tabular}{l | l      l   l      l        l    l     l   l    l     l }  
$x\bullet_{5,1,3} z$ & 0 & 1 & 2 & 3 & 4 & 5   \\
\hline
0 & 15 & 3 & 3 & 3 & 3 &  15 \\
1 & 5 & 1 & 1 & 1 & 1 &  5 \\
2 & 5 & 1 & 1 & 1 & 1&    5 \\
3 & 15 & 3 & 3 & 3 & 3 &  15 \\
\end{tabular}
\end{center}
The extracted table for the divisors of $15$ is
\begin{center}
\begin{tabular}{l | l      l   l      l        l    l     l   l    l     l }  
$x\bullet_{5,1,3} z$ & 1 &  3 & 5 & 15  \\
\hline
1 & 1 & 1 & 5 & 5 \\
3 & 3 & 3 & 15 & 15 \\
5 & 1 & 1 & 5 & 5 \\
15 & 3 & 3 & 15 & 15\\
\end{tabular}
\end{center}
which agrees with the ``dual'' table of divisors of $15$ for
$(5,15,3)$. 
\end{example}

\begin{example}
Let $y=2$. First the table for the commutative product $\bullet_{3,2,3}$,
having line and column period $3$ and square period $6$.
The element $3$ is absorbing (constant line and column $3$).
\begin{center}
\begin{tabular}{l | l      l   l      l        l    l     l   l    l     l }  
$x\bullet_{3,2,3} z$ & 0 & 1 & 2 & 3 & 4 & 5 & 6 \\
\hline
0 & 6 & 3 & 6 & 3 & 6 & 3 & 6  \\
1 & 3 & 1 & 1 & 3 & 1 & 1 & 3 \\
2 & 6 & 1 & 2 & 3 & 2 & 1 & 6 \\
3 & 3 & 3 & 3 & 3 & 3 & 3 & 3 \\
4 & 6 & 1 & 2 & 3 & 2 & 1 & 6 \\
5 & 3 & 1 & 1 & 3 & 1 & 1 & 3 \\
6 & 6 & 3 & 6 & 3 & 6 & 3 & 6  
\end{tabular}
\end{center}
The table for the non-commutative product $\cdot_{3,2,4}$
having line period $4$ and column period $6$ and square period $12$
has been given in the introduction. 
As it turns out, there is an even shorter  period $3$ concerning columns.

Next, the table for the non-commutative product $\cdot_{3,2,5}$
having line period $10$ and column period $6$ and square period $30$:
\begin{center}
\begin{tabular}{l | l      l   l      l        l    l     l   l    l     l }  
$x\cdot_{3,2,5} z$ & 0 & 1 & 2 & 3 & 4 & 5 & 6  \\
\hline
0 & 30 & 5 & 10 & 15 & 10 & 5 & 30  \\
1 & 3 & 1 & 1 & 3 & 1 & 1 & 3 \\
2 & 6 & 1 & 2 & 3 & 2 & 1 & 6 \\
3 & 3 & 1 & 1 & 3 & 1 &  1 & 3 \\
4 & 6 & 1 & 2 & 3 & 2 & 1 & 6 \\
5 & 15 & 5 & 5 & 15 & 5 & 5 & 15 \\
6 & 6 & 1 & 2 & 3 & 2 & 1 & 6 \\
7 & 3 & 1 & 1 & 3 & 1 &  1 & 3 \\
8 & 6 & 1 & 2 & 3 & 2 & 1 & 6 \\
9 & 3 & 1 & 1 & 3 & 1 &  1 & 3 \\
10 & 30 & 5 & 10 & 15 & 10 & 5 & 30  
\end{tabular}
\end{center}
The same table gives rise to the multiplication table for the set of divisors of $30$:
\begin{center}
\begin{tabular}{l | l      l   l      l        l    l     l   l    l     l }  
$x\cdot_{3,2,5} z$ & 1 & 2 & 3 & 5 & 6 & 10  & 15 & 30   \\
\hline
1 & 1 & 1 & 3 & 1 & 3  & 1 & 3 & 3  \\
2 & 1 & 2 & 3 & 1 & 6 & 2 & 3 & 6 \\
3 & 1 & 1 & 3 & 1 & 3 & 1 & 3 & 3 \\
5 & 5 & 5 & 15 &  5 &15& 5 &15 &15\\
6 & 1 & 2 & 3 & 1 & 6 & 2 &3 & 6 \\
10&5 &10&15&5& 30 &10&15&30\\
15&5 & 5 &15&5&15&5 & 15 & 15 \\
30&5 &10&15& 5 &30&10 & 15&30\\
\end{tabular}
\end{center}
\end{example}


\begin{example} 
Here an example where $a,y,b$ are composed numbers having
common factors:

\begin{center}
\begin{tabular}{l | l      l   l      l        l    l     l   l    l     l }  
$x\cdot_{2,6,3} z$ & 0 & 1 & 2 & 3 & 4 & 5   \\
\hline
0 & 6 & 3 & 6 & 3& 6 & 3 
\\
1 & 2 & 1 & 2 & 1 & 2 & 1
\\
2 & 2 & 1 & 2 & 1 & 2 & 1
\\
3 & 6 & 3 & 6 & 3& 6 & 3 
\\
4 & 2 & 1 & 2 & 1 & 2 & 1
\\
5 & 2 & 1 & 2 & 1 & 2 & 1
\end{tabular}
\end{center}
and the extracted table for the divisors of 6 is:
\begin{center}
\begin{tabular}{l | l      l   l      l        l    l     l   l    l     l }  
$x\cdot_{2,6,3} z$ & 1 & 2 & 3 & 6    \\
\hline
1 & 1 & 2 & 1 & 2
\\
2 & 1 & 2 & 1 & 2
\\
3 & 3 & 6 & 3 & 6
\\
6 & 3 & 6 & 3 & 6
\end{tabular}
\end{center}
When $a,y,b$ have all three non-trivial factors in common, the structure
of the products becomes quite complicated, 
and certainly deserves to be studied further.
\end{example}

\begin{example}
Recall from Theorem \ref{th:bottom-top} that the 
non-principal parastrophes may have a ``bottom'', or not.
When the bottom is empty, then the tables are non-periodic. 
This is the case, for instance, for the product $(a,z) \mapsto L(8,a,4,2,z)$:
\begin{center}
\begin{tabular}{l | l      l   l      l        l    l     l   l    l     l }  
$a \backslash z$ 
    &0 & 1 & 2 & 3 & 4 & 5 & 6  & 7 & 8\\
\hline
0 & 0 & 2 & 2 & 6 & 4& 10& 6 &14 &8\\ 
1 & 4 & 2 & 2 & 2 & 4& 2 & 2 & 2 & 4 \\
2 & 4 & 2 & 2 & 2 & 4& 2 & 2 & 2 & 4 \\
3& 12& 2 & 2 & 6 & 4& 2 & 6 & 2 & 4 \\
4 & 4 & 2 & 2 & 2 & 4& 2 & 2 & 2 & 4 \\
5 &20 &2 & 2 & 2 & 4&10&2 & 2 & 4 \\
6 & 12 &2 &2 &4  & 6 &2 & 6 & 2 & 4 \\
7 & 28 &2 & 2&2 & 4 & 2 & 2&14& 4 \\
\end{tabular}
\end{center}
There is no easily visible pattern how to continue this table.
The parastrophic product
$(b,z) \mapsto L(8,2,4,b,z)$ is periodic and seems to
follow a pattern similar to a ``principal'' product: 
\begin{center}
\begin{tabular}{l | l      l   l      l        l    l     l   l    l     l }  
$b \backslash z$ 
    &0 & 1 & 2 & 3 & 4 & 5 & 6  & 7 & 8\\
\hline
0 & 8 & 8 & 8 & 8 & 8 & 8 & 8 & 8 & 8 \\
1 & 4 & 1 & 2 & 1 & 4 & 1 & 2 & 1 & 4 \\
2 & 4 & 2 & 2 & 2 & 4 & 2 & 2 & 2 & 4 \\
3 & 4 & 1 & 2 & 1 & 4 & 1 & 2 & 1 & 4 \\
4 & 4 & 4 & 4 & 4 & 4 & 4 & 4 & 4 & 4 \\
5 & 4 & 1 & 2 & 1 & 4 & 1 & 2 & 1 & 4 \\
6 & 4 & 2 & 2 & 2 & 4 & 2 & 2 & 2 & 4 \\
7 & 4 & 1 & 2 & 1 & 4 & 1 & 2 & 1 & 4 \\
8 & 8 & 8 & 8 & 8 & 8 & 8 & 8 & 8 & 8 \\
\end{tabular}
\end{center}
\end{example}

These constructions give a considerable number of finite semigroups, and hence seem to be a quite effective
machinery to ``produce'' semigroups.


\section{Afterthoughts}

In several respects, the arithmetic case is an ``antipode'' of the  geometric
approach developed in \cite{BeKi1, BeKi2, BeKi12}. Namely,
from a geometric point of view, the arithmetic case is a ``zero dimensional projective geometry'':  if $\K$ is a field, then $\Gras_\K(\K)$ would be the 
zero-dimensional projective space $\K\PP^0$, which is uninteresting.
If $\K$ is a ring, notably, for $\K=\Z$,  things change considerably:
\begin{enumerate}
\item 
{\em transversality} is a major tool used in \cite{BeKi1}; but
  the arithmetic case does not admit any transversal pairs (two non-trivial
submodules of $\Z$ always have non-trivial intersection)! 
\item
 {\em anti-automorphisms} are crucial ingredients to define ``classical
 geometries'' (see \cite{BeKi2}), but in the arithmetic case there are no
 antiautomorphisms, 
\item 
 of course, one can define abstractly a ``dual geometry'' (just by taking
 the dual lattice and exchanging $a$ and $b$); but at a first glance 
 there is no ``natural model'' how to ``realize this dual geometry in nature'',
\item
 {\em group actions with ``big'' (``open'') orbits} are heavily used in \cite{BeKi1};
 but in the arithmetic case,
the geometry is {\em highly non-homogeneous}: the  linear group $\GL(1,\Z) = \{ \pm 1 \}$ has trivial orbits in $\N$, and hence group-theoretic arguments
are of no use. 
\end{enumerate}
Thus, from a geometric viewpoint, 
the arithmetic case seems to be difficult to understand. 
As we have seen, 
this is partially compensated by the simple formula 
$L = \Gamma = U$; but a deeper understanding of the link between the
geometric and the arithmetic aspects would be desirable.


\end{document}